\documentclass[
final
]{dmtcs-episciences}
\usepackage[utf8]{inputenc}
\usepackage{amsmath,amsthm}			
\usepackage{amssymb}
\usepackage{accents}
\usepackage{hyperref}				
\usepackage{tikz}
\usepackage{enumerate,dsfont}
\usepackage{xurl}
\usepackage[all]{xy}
\usepackage[enableskew,vcentermath]{youngtab}
\usepackage{cleveref}
\usetikzlibrary{shapes.geometric,arrows}
\usetikzlibrary{calc,patterns}

\newtheorem{theorem}{Theorem}[section]
\newtheorem{lemma}[theorem]{Lemma}
\newtheorem{corollary}[theorem]{Corollary}

\theoremstyle{definition}
\newtheorem{example}[theorem]{Example}  
\newtheorem{definition}[theorem]{Definition}   

\theoremstyle{remark}

\tikzstyle{vertex}=[circle, draw, fill=black, inner sep=0pt, minimum size=1pt]

\newcommand{\G}{\mathfrak{G}}
\renewcommand{\top}{\textbf{top}}
\newcommand\thickbar[1]{\accentset{\rule{.4em}{.8pt}}{#1}}
\definecolor{c1}{HTML}{0085FF}
\definecolor{c2}{HTML}{FF9000}
\definecolor{c3}{HTML}{90D300}
\definecolor{c4}{HTML}{9728A1}

\newcommand{\Aut}{\operatorname{Aut}}

\title{Coloring Groups}
\author[Ben Adenbaum and Alexander Wilson]{Ben Adenbaum\affiliationmark{1} 
\and 
Alexander Wilson\affiliationmark{2}
}
\affiliation{
  Department of Mathematics, Dartmouth College, Hanover, NH, USA\\
  Department of Mathematics, Oberlin College, Oberlin, OH, USA}
  \keywords{Groups from graphs, toggle group, independence poset}
\title[Coloring Groups]{Coloring Groups}

\begin{document}
\publicationdata{vol. 26.2}{2024}{9}{10.46298/dmtcs.12753}{2023-12-28; None}{2024-05-07}

\maketitle
\hspace{10pt}

\begin{abstract}
We introduce coloring groups, which are permutation groups obtained from a proper edge coloring of a graph. These groups generalize the generalized toggle groups of Striker (which themselves generalize the toggle groups introduced by Cameron and Fon-der-Flaass). We present some general results connecting the structure of a coloring group to the structure of its graph coloring, providing graph-theoretic characterizations of the centralizer and primitivity of a coloring group. We apply these results particularly to generalized toggle groups arising from trees as well as coloring groups arising from the independence posets introduced by Thomas and Williams. \end{abstract}

\section{Introduction}\label{sec:intro}

Cameron and Fon-der-Flaass in ~\cite{cameron1995orbits} first introduced toggle groups on order ideals in order to study the map now known as rowmotion (the name ``toggle group'' was introduced later in \cite{striker2012promotion}). Given a poset $P$, each element $p\in P$ has a corresponding permutation $\tau_p$ (called a toggle) of the order ideals $\mathcal{J}(P)$ of $P$ which acts in the following way, toggling the element $p$ as long as the result is still an order ideal.
\begin{align*}
    \tau_p(I)&=\begin{cases}
        I\cup\{p\} & \text{ if } p\notin I\text{ and }I\cup\{p\}\in\mathcal{J}(P)\\
        I\setminus\{p\} & \text{ if } p\in I\text{ and }I\setminus\{p\}\in\mathcal{J}(P)\\
        I & \text{ otherwise}
    \end{cases}
\end{align*}

To understand properties of rowmotion and associated operations---in particular their order---Striker introduced generalized toggle groups \cite{striker2018rowmotion} where the order ideals of a poset are replaced with an arbitrary set of allowable subsets and the poset is replaced by just a finite set. 
\begin{definition}[Toggle Group~\cite{striker2018rowmotion}]

    Let $E$ be a finite set, and let $L$ be any subset of the power set $2^E$. For each $e\in E$, the toggle $\tau_e$ is defined as follows:
    \[\tau_e(X) = \begin{cases}
        X\cup \{e\} & \text{if } e\notin X \text{ and } X\cup \{e\}\in L\\
         X\setminus \{e\} & \text{if } e\in X \text{ and } X\setminus \{e\}\in L\\
         X & \text{otherwise}
    \end{cases}\]

    The set $\{\tau_e | e\in E\}$ are the set of toggles.  Let $T(L)$ be the subgroup of $S_L$ generated by the toggles. 
\end{definition}
\begin{example}\label{ex:gen_tog}

        Let $E = [4]$ and $L = \{\emptyset, \{1\}, \{1,2\}, \{1,2,3\}, \{1,2,3,4\}, \{4\}, \{3,4\}, \{2,3,4\}\}$
        \begin{itemize}
        \item $\tau_{1} = (\emptyset, \{1\})(\{2,3,4\}, \{1,2,3,4\})$
        \item $\tau_{2} = (\{1\}, \{1,2\})(\{3,4\}, \{2,3, 4\})$
        \item $\tau_{3} = (\{1,2\}, \{1,2,3\})(\{4\}, \{3,4\})$
        \item $\tau_{4} = (\emptyset, \{4\})(\{1,2,3\}, \{1,2,3,4\})$
        \end{itemize}
          In this case  $T_L  \cong S_4 \rtimes (C_2)^3$
    \end{example}
The idea of order ideal toggling has been generalized from the purely combinatorial perspective to the piece-wise linear and birational realms in~\cite{einstein2021combinatorial} and~\cite{grinberg2016iterative,grinberg2015iterative} respectively. These groups were studied in-depth by Striker in~\cite{striker2018rowmotion} when arising from some families of combinatorial and geometric objects. More recently, by Bloom and Saracino in~\cite{bloom2023transitive}, they have been considered with respect to their group theoretic properties such as primitivity.

To all generalized toggle groups, there is a natural partial order called the \emph{toggle poset} which we denote by $P_L$. The cover relations in this poset are given by $X \lessdot Y$ if $X\subseteq Y$ and there exists some $e\in E$ such that $\tau_e(X) = Y$, and $|Y|-|X| =1$.  Importantly, each edge of the Hasse diagram is naturally labeled by the element toggled in for the cover.

    \begin{example}\label{ex:gen_tog_poset}
    The toggle poset of the generalized toggle group of Example~\ref{ex:gen_tog}
    \centering

        \begin{tikzpicture}
        
            \node (1) at (0,0) {$\emptyset$};
            \node (2) at (1,1) {$\{4\}$};
            \node (3) at (-1,1) {$\{1\}$};
            \node (4) at (-1,2) {$\{1,2\}$};
            \node (5) at (1,2) {$\{3,4\}$};
            \node (6) at (1,3) {$\{2,3,4\}$};
            \node (7) at (-1,3) {$\{1,2,3\}$};
            \node (8) at (0,4) {$\{1,2,3,4\}$};
            \draw [-] (1) --node[below right]{4} (2);
            \draw [-] (1) --node[below left]{1} (3);
            \draw [-] (5) --node[right]{3} (2);
            \draw [-] (3) --node[left]{2} (4);
            \draw [-] (5) --node[right]{2} (6);
            \draw [-] (4) --node[left]{3} (7);
            \draw [-] (8) --node[left]{4}  (7);
            \draw [-] (8) --node[right]{1} (6);
        \end{tikzpicture}
    \end{example}

In this paper, we take inspiration from this edge labeling of the Hasse diagram to introduce coloring groups, which arise from proper edge colorings of finite graphs. Coloring groups encompass all finite groups generated by involutions (in particular the generalized toggle groups). Because our interest lies in understanding combinatorial objects with natural involution structure (e.g. toggle groups, or quivers up to mutation), we want to understand how the structure of the graphs and their edge colorings control the structure of the corresponding coloring group.

The paper is organized as follows:

\begin{itemize}
    \item In \Cref{sec:prelim} we define coloring groups and characterize which permutation groups arise as coloring groups.
    \item In \Cref{sec:tools} we prove general tools for working with coloring groups as well as more specialized tools in the case that the underlying graph $G$ is a tree.
    \item In \Cref{sec:applications} we study generalized toggle groups arising from trees as well as coloring groups arising in the study of independence posets.
    \item In \Cref{sec:constructions} we provide some constructions of the groups $S_n, A_n, D_n,$ and the type-B coxeter group $B_m$ as coloring groups.
    \item Finally, in \Cref{sec:data} we provide data on the coloring groups arising from trees of order at most 12.
\end{itemize}

\section{Preliminaries and Definitions}\label{sec:prelim}

In this paper, a graph of order $n$ will be a simple undirected graph on $n$ vertices unless otherwise specified. Given a graph $G$, a proper edge coloring of $G$ on $k$ colors is a surjective map $\kappa:E(G)\to[k]$ such that for any two incident edges $e_1$ and $e_2$, we have that $\kappa(e_1)\neq\kappa(e_2)$. We define the coloring group $\G_\kappa$ corresponding to $\kappa$ as the subgroup of $S_{V(G)}
\cong S_n$ generated by the involutions \begin{align}
    \tau_a=\prod_{\substack{\{i,j\}\in E(G)\\\kappa(\{i,j\})=a}} (i, j)
\end{align} for $a\in[k]$. Note that because the edge coloring is proper, the product can be taken in any order, and $\tau_a$ simply swaps the vertices on either end of each edge colored $a$.

\begin{example}\label{ex:gl32} The following graph is shown with a proper edge coloring $\kappa:E(G)\to[3]$ where the colors 1, 2, and 3 of the edges are represented by \textbf{\textcolor{c1}{blue}}, \textbf{\textcolor{c2}{orange}}, and \textbf{\textcolor{c3}{green}} respectively.
    
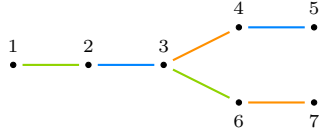
\begin{center}
    \begin{tikzpicture}[scale = .5]
        \node (Label) at (-9, .5) {\scriptsize$1$};
        \node (1) at (-9, 0) {};
        \node (Label) at (-7, .5) {\scriptsize$2$};
        \node (2) at (-7, 0) {};
        \node (Label) at (-5, .5) {\scriptsize$3$};
        \node (3) at (-5, 0) {};
        \node (Label) at (-3, 1.5) {\scriptsize$4$};
        \node (4) at (-3, 1) {};
        \node (Label) at (-1, 1.5) {\scriptsize$5$};
        \node (5) at (-1,1) {};
        \node (Label) at (-3, -1.5) {\scriptsize$6$};
        \node (6) at (-3, -1) {};
        \node (Label) at (-1, -1.5) {\scriptsize$7$};
        \node (7) at (-1, -1) {};
        \draw[color = c3, thick] (1) -- (2);
        \draw[color = c1, thick] (3) -- (2);
        \draw[color=c2,thick] (3) -- (4);
        \draw[color=c1,thick] (4) -- (5);
        \draw[color=c3, thick] (3) -- (6);
        \draw[color=c2, thick] (6) -- (7);
        \draw[color = black, fill] (1) circle (2pt);
        \draw[color = black, fill] (2) circle (2pt);
        \draw[color = black, fill] (3) circle (2pt);
        \draw[color = black, fill] (4) circle (2pt);
        \draw[color = black, fill] (5) circle (2pt);
        \draw[color = black, fill] (6) circle (2pt);
        \draw[color = black, fill] (7) circle (2pt);
    \end{tikzpicture}
\end{center}

The corresponding coloring group is isomorphic to the general linear group $GL(3,2)$. \begin{align*}
    \G_\kappa=\langle\textcolor{c1}{(2,3)(4,5)}, \textcolor{c2}{(3,4)(6,7)},\textcolor{c3}{(1,2)(3,6)}\rangle\cong GL(3,2).
\end{align*}
\end{example}

As a subgroup of $S_n$, the group $\G_\kappa$ is an example of a permutation group of degree $n$. Cayley's theorem tells us that every finite group is isomorphic to a permutation group, but coloring groups are in a smaller class of groups which are generated by involutions. \Cref{thm:all_groups} shows that the class of coloring groups is in fact precisely the class of finite groups generated by involutions.

\begin{theorem}\label{thm:all_groups}
Let $\mathfrak{G}$ be a finite group generated by involutions. Then there is a connected graph $G$ with proper edge coloring $\kappa$ such that $\mathfrak{G}_\kappa$ is isomorphic to $\mathfrak{G}$.
\end{theorem}
\begin{proof}
    Suppose that $\mathfrak{G} = \langle g_1, \dots, g_k\rangle$ with each $g_i$ an involution. Let $G$ be the Cayley graph of $\mathfrak{G}$. If $ \{a,b\}\in E(G)$ then $b = a g_i$ for a unique generator $g_i$, so define $\kappa(\{a,b\})=i$. As the generators are involutions, this edge coloring is proper. Let $\G_\kappa$ be the corresponding coloring group. We claim that $\mathfrak{G}_\kappa\cong \mathfrak{G}$. To see why, consider $\mathfrak{G}$ as as a subgroup of $S_{V(G)}$. Note that each $g_i$ corresponds to the product of the disjoint transpositions $(a,b)$ such that $b = a g_i$, so $\mathfrak{G}$ is the group generated by these involutions. For each $i$ under this identification, $g_i$ is exactly $\tau_i$, so the groups are equal as permutation groups. 
\end{proof}

This is a very large class of groups, and the graphs obtained in the proof of \Cref{thm:all_groups} have as many vertices as the group $\G$ has elements. We thus restrict ourselves to situations where the graph $G$ and its proper edge coloring $\kappa$ have nice structure. Of particular interest will be the symmetry of this coloring, captured in the following definition.

\begin{definition}
    For a proper edge coloring $(G,\kappa)$, the group $\Aut_\kappa(G)$ consists of all permutations $\sigma\in S_{V(G)}$ so that $\{i,j\}$ is an edge of $G$ colored $a$ under $\kappa$ if and only if $\{\sigma(i),\sigma(j)\}$ is an edge of $G$ colored $a$ under $\kappa$.
\end{definition}

\begin{lemma}\label{lem:fixing_a_vertex} For a proper edge coloring $(G,\kappa)$, if $\sigma\in \Aut_\kappa(G)$ fixes a vertex $v$, it fixes every neighbor of $v$.
\end{lemma}

\begin{proof}
    Suppose $\sigma\in\Aut_\kappa(G)$ fixes $v$ and $w$ is a neighbor of $v$ connected by an edge colored $a$. Then $\{v,\sigma(w)\}$ must be an edge colored $a$. Because $\kappa$ is a proper edge coloring, the only edge incident to $v$ colored $a$ is the edge $\{v,w\}$. Hence $\sigma(w)=w$.
\end{proof}

A permutuation group $\G$ of degree $n$ acts naturally on the set partitions of $[n]$. If $\G$ fixes any nontrivial set partition $\pi$ (i.e. one not consisting of a single block or all singletons), then we say that $\G$ is imprimitive with $\pi$ being a system of imprimitivity. Otherwise, we say that $\G$ is primitive.

\section{Tools for Analyzing Coloring Groups}\label{sec:tools}

We begin with some results that can be applied to any coloring group. The following theorem characterizes the centralizer of a coloring group.

\begin{theorem}\label{thm:centralizer_of_coloring}
    For proper edge coloring $(G,\kappa)$, the centralizer $C_{S_{V(G)}}(\mathfrak{G}_\kappa)$ of the corresponding coloring group in the symmetric group is isomorphic to $\Aut_\kappa(G)$.
\end{theorem}

\begin{proof}
    Let $\sigma\in C_{S_{V(G)}}(\mathfrak{G}_\kappa)$ and suppose that $\{i,j\}$ is an edge colored $a$. Then $(i,j)$ is a transposition in $\tau_a$. Because $\sigma$ commutes with $\tau_a$, \[\sigma\tau_a\sigma^{-1}=\tau_a.\] Then $(\sigma(i),\sigma(j))$ is also a transposition in $\tau_a$ and hence $\{\sigma(i),\sigma(j)\}$ is an edge of $G$ colored $a$ under $\kappa$. Hence, $\sigma\in\Aut_\kappa(G)$.

    Suppose conversely that $\sigma\in\Aut_\kappa(G)$. Then for each color, $\sigma$ induces a permutation $\sigma'$ of the edges of that color. Fixing a color $a$, let $\{i_1,j_1\},\dots,\{i_m,j_m\}$ be the edges of that color. Then \begin{align*}
        \sigma\tau_a\sigma^{-1}&=(\sigma(i_1),\sigma(j_1))\dots(\sigma(i_m),\sigma(j_m))\\
        &=(i_{\sigma'(1)},j_{\sigma'(1)})\dots(i_{\sigma'(m)},j_{\sigma'(m)})\\
        &=\tau_a
    \end{align*} where the final equality follows from the fact that $\kappa$ is a proper edge coloring, so these transpositions commute pairwise. Because $\sigma$ commutes with each generator of $\mathfrak{G}_\kappa$, we have that $\sigma\in C_{S_{V(G)}}(\mathfrak{G}_\kappa)$
\end{proof}

\begin{corollary}
    For a proper edge coloring $(G,\kappa)$, \[\mathfrak{G}_\kappa \subseteq C_{S_{V(G)}}(\Aut_\kappa(G)).\] In particular, a necessary condition for $\mathfrak{G}_\kappa\cong S_{V(G)}$ is that $\Aut_\kappa(G)$ is trivial.
\end{corollary}

The following definition and lemma provide a graph-theoretic characterization of when a coloring group is primitive via vertex colorings.

\begin{definition}
    For a proper edge coloring $(G,\kappa)$, consider a vertex coloring of $\nu:V(G)\to[\ell]$. The coloring $\nu$ is called imprimitive if the following two statements hold:
    \begin{enumerate}
        \item[(i)] If an edge colored $b$ connects a vertex colored $a$ to a vertex colored $c$ with $a\neq c$, then every vertex colored $a$ is connected to a vertex colored $c$ by an edge colored $b$.
        \item[(ii)] At least one color appears strictly more than one time and strictly less than $|G|$ times.
    \end{enumerate}
\end{definition}

\begin{example}\label{ex:imprimitive_path} The following edge coloring of the path graph $P_{15}$ has an imprimitive vertex coloring.
    
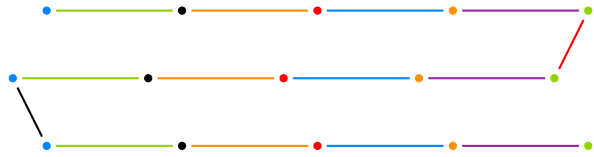
\begin{center}
        \begin{tikzpicture}[scale = .45]
            \node (1a) at (-4, -2) {};
            \node (1b) at (-5, -4) {};
            \node (1c) at (-4, -6) {};
            \node (2a) at (0, -2) {};
            \node (2b) at (-1, -4) {};
            \node (2c) at (0, -6) {};
            \node (3a) at (4, -2) {};
            \node (3b) at (3, -4) {};
            \node (3c) at (4, -6) {};
            \node (4a) at (8, -2) {};
            \node (4b) at (7, -4) {};
            \node (4c) at (8, -6) {};
            \node (5a) at (12, -2) {};
            \node (5b) at (11, -4) {};
            \node (5c) at (12, -6) {};
            \draw[color = c1, fill] (1a) circle (3pt);
            \draw[color = c1, fill] (1b) circle (3pt);
            \draw[color = c1, fill] (1c) circle (3pt);
            \draw[color = black, fill] (2a) circle (3pt);
            \draw[color = black, fill] (2b) circle (3pt);
            \draw[color = black, fill] (2c) circle (3pt);
            \draw[color = red, fill] (3a) circle (3pt);
            \draw[color = red, fill] (3b) circle (3pt);
            \draw[color = red, fill] (3c) circle (3pt);
            \draw[color = c2, fill] (4a) circle (3pt);
            \draw[color = c2, fill] (4b) circle (3pt);
            \draw[color = c2, fill] (4c) circle (3pt);
            \draw[color = c3, fill] (5a) circle (3pt);
            \draw[color = c3, fill] (5b) circle (3pt);
            \draw[color = c3, fill] (5c) circle (3pt);
            \draw[color = c3, thick] (1a) -- (2a);
            \draw[color = c2, thick] (3a) -- (2a);
            \draw[color = c1, thick] (3a) -- (4a);
            \draw[color = c4, thick] (5a) -- (4a);
            \draw[color = red, thick] (5a) -- (5b);
            \draw[color = c4, thick] (5b) -- (4b);
            \draw[color = c1, thick] (4b) -- (3b);
            \draw[color = c2, thick] (3b) -- (2b);
            \draw[color = c3, thick] (2b) -- (1b);
            \draw[color = black, thick] (1b) -- (1c);
            \draw[color = c3, thick] (1c) -- (2c);
            \draw[color = c2, thick] (3c) -- (2c);
            \draw[color = c1, thick] (3c) -- (4c);
            \draw[color = c4, thick] (5c) -- (4c);
        \end{tikzpicture}
    \end{center}
\end{example}

\begin{example}\label{ex:imprimitive_tree} The following proper edge coloring has an imprimitive vertex coloring where the vertices are partitioned into the three sets of four vertices which are on the same line.
    
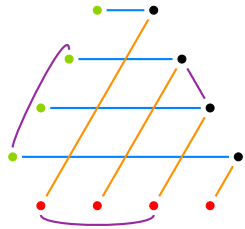
\begin{center}
        \begin{tikzpicture}[scale = .75]
        \node (1) at (2, 1.734/2) {};
        \node (2) at (1.5, 3.468/2) {};
        \node (3) at (1, 5.202/2) {};
        \node(4) at (0.5, 6.936/2) {};
        \node (5) at (-2, 1.734/2) {};
        \node (6) at (-1.5, 3.468/2) {};
        \node (7) at (-1, 5.202/2) {};
        \node (8) at (-0.5, 6.936/2) {};
        \node (9) at (-2.5+1, 0) {};
        \node (10) at (-2.5+2, 0) {};
        \node (11) at (2.5-2, 0) {};
        \node (12) at (2.5-1, 0) {};
        \draw[color = c1, thick] (1) -- (5);
        \draw[color = c1, thick] (2) -- (6);
        \draw[color = c1, thick] (3) -- (7);
        \draw[color = c1, thick] (4) -- (8);
        \draw[color = c2, thick] (1) -- (12);
        \draw[color = c2, thick] (2) -- (11);
        \draw[color = c2, thick] (3) -- (10);
        \draw[color = c2, thick] (4) -- (9);
        \draw[color = c4, thick] (2) -- (3);
        \draw[color = c4, thick] (5) .. controls +(0,.4) and +(0,.4*1.7) .. (7);
        \draw[color = c4, thick] (9) .. controls +(0,-.4) and +(0,-.4) .. (11);
        \draw[color = black, fill] (1) circle (2pt);
        \draw[color = black, fill] (2) circle (2pt);
        \draw[color = black, fill] (3) circle (2pt);
        \draw[color = black, fill] (4) circle (2pt);
        \draw[color = c3, fill] (5) circle (2pt);
        \draw[color = c3, fill] (6) circle (2pt);
        \draw[color = c3, fill] (7) circle (2pt);
        \draw[color = c3, fill] (8) circle (2pt);
        \draw[color = red, fill] (9) circle (2pt);
        \draw[color = red, fill] (10) circle (2pt);
        \draw[color = red, fill] (11) circle (2pt);
        \draw[color = red, fill] (12) circle (2pt);
    \end{tikzpicture}
    \end{center}
\end{example}

\begin{lemma} For a proper edge coloring $(G,\kappa)$, the group $\mathfrak{G}_\kappa$ is imprimitive if and only if there exists an imprimitive vertex coloring of $G$ with respect to the edge coloring $\kappa$.
\end{lemma}

\begin{proof} Given an imprimitive vertex coloring, the set partition obtained by putting two vertex labels in the same set if and only if the vertices are colored with the same color is a system of imprimitivity. Conversely, a system of imprimitivity prescribes the distribution of vertex colors for an imprimitive vertex coloring.
\end{proof}

When considering toggle groups, one of the motivating questions has been whether elements that mimic coxeter elements, i.e. products of the generating involutions in some prescribed order, have a predictable order. When we restrict to the case of forests, we get the following.
\begin{lemma}\label{lem:long_cycle} Suppose $G$ is a disjoint union of $\ell$ trees with orders given by $\lambda_1\geq\lambda_2\geq\dots\geq\lambda_\ell$. Then for any proper edge coloring $\kappa:E(G)\to [k]$, the product of all the corresponding generators \[\tau_1\dots\tau_k\in \mathfrak{G}_\kappa\] in any order has cycle type $\lambda$.
\end{lemma}

\begin{proof} It suffices to show that in the case that $G$ is a tree of order $n$, the product of all the transpositions corresponding to edges in $G$ in any order is an $n$-cycle. The following proof is based on an argument by Darij Grinberg communicated in a comment on Math Stack Exchange\footnote{\url{https://math.stackexchange.com/questions/2577311/product-of-transpositions-from-edges-of-a-tree-is-a-cycle-of-length-n}}.

We proceed by induction on the order of $G$. Let $\tau_1,\dots,\tau_{n-1}$ be the transpositions corresponding to the $n-1$ edges of $G$ in any order. By conjugating by a transposition on one end of this sequence, we can cyclically permute the sequence without changing the cycle structure of the product. Without loss of generality, assume we have cyclically permuted the transpositions so that the transposition $\tau_1$ appearing first corresponds to an edge incident to a leaf labeled $x$. By the inductive hypothesis, $\tau_2\tau_3\cdots\tau_{n-1}$ is an $n-1$-cycle on $\{1,\dots,n\}\setminus\{x\}$. Hence, $\tau_1\tau_2\dots\tau_{n-1}$ is an $n$-cycle on $\{1,\dots,n\}$.
\end{proof}

In particular when $G$ is a tree of order $n$ any such element is an $n$-cycle, which will be essential for Theorem~\ref{thm:primitive_trees} and its consequences. 

\begin{lemma}\label{lem:size_bound}
    If $(G,\kappa)$ is a proper edge-coloring of a tree with $n$ vertices that uses at least 3 colors, \[\left|\G_\kappa\right|\geq\frac{n^2}{n-\varphi(n)}\] where $\varphi(n)$ is the Euler totient function.
\end{lemma}

\begin{proof}
    Note that if $\sigma_1$ and $\sigma_2$ are distinct long cycles in $S_n$ which generate the same subgroup, then $\sigma_1={\sigma_2}^r$ for some $1<r<n$. Then because ${\sigma_2}^s$ has a fixed point if and only if $n|s$, \[\sigma_1(a)={\sigma_2}^r(a)={\sigma_2}^{r-1}(\sigma_2(a))\neq{\sigma_2}(a).\]

    Without loss of generality, suppose that $G$ has a leaf labeled $1$ incident to an edge colored $1$. Then any long cycle of the form \begin{align*}
        \sigma=\tau_1 \tau_{i_1}\dots\tau_{i_{k-1}}
    \end{align*} has the property that $\sigma(1)=\tau_1(1)$. Because there are at least three colors there must be at least two colors that don't commute with each other. Hence, there are at least two long-cycles $\sigma_1$ and $\sigma_2$ of this form. Because $\sigma_1(1)=\sigma_2(1)$, they must generate distinct subgroups. Then $\G_\kappa$ has a subset $\langle\sigma_1\rangle\langle\sigma_2\rangle$ with order \[\left|\langle\sigma_1\rangle\langle\sigma_2\rangle\right|=\frac{\left|\langle\sigma_1\rangle\right|\cdot\left|\langle\sigma_2\rangle\right|}{\left|\langle\sigma_1\rangle\cap\langle\sigma_2\rangle\right|}.\] Note that because $\sigma_1$ and $\sigma_2$ generate distinct subgroups, their intersection can only contain elements of the form ${\sigma_1}^m$ where $m$ is not relatively prime to $n$. Hence $\left|\langle\sigma_1\rangle\cap\langle\sigma_2\rangle\right|\leq n-\varphi(n)$.
\end{proof}

The following theorem provides a condition for narrowing down the possible primitive coloring groups that can arise from a graph $G$.

\begin{theorem}\label{thm:primitive_trees}
    Let $\kappa:E(G)\to[k]$ be a proper edge coloring of a graph of order $n$ and suppose that $\G_\kappa$ is primitive. Suppose that restricting to the edges colored by a subset of $\ell$ colors (and their incident vertices) forms a tree $T$ of order $m$. If $n-m\geq 3$, then $A_n\leq \G_\kappa$. Furthermore, if we assume that $\ell>2$, then $A_n\leq \G_\kappa$ unless one of the following hold (where in each case $q$ is a prime power):
    \begin{itemize}
        \item[(i)] $n-m=2$,\;$n=q+1$,\; $PGL_2(q)\leq \G_\kappa\leq P\Gamma L_2(q)$, or
        \item[(ii)] $n-m=1$ and either
        \begin{itemize}
            \item[(a)] $n=q^d$,\;$d\geq1$,\;$AGL_d(q)\leq\G_\kappa\leq A\Gamma L_d(q)$, or
            \item[(c)] $n=24$,\;$\G_\kappa=M_{24}$, or
        \end{itemize}
        \item[(iii)] $n-m=0$ and either
        \begin{itemize}
            \item[(a)] $n=\frac{q^d-1}{q-1}$,\;$d\geq2$,\;$PGL_d(q)\leq\G_\kappa\leq P\Gamma L_d(q)$, or
            \item[(b)] $n=23$,\;$\G_\kappa=M_{23}$.
        \end{itemize}
    \end{itemize}
\end{theorem}

\begin{proof}
    Much of the heavy-lifting in this proof is done by noting that the existence of such a tree $T$ guarantees a cycle in $\G_\kappa$ that fixes $n-m$ points using \Cref{lem:long_cycle}, and then applying \cite[Theorem 1.2]{jones2014primitive}. We need only rule out the following cases:
    \begin{itemize}
        \item $n-m=1,\; p\geq5\text{ is prime,}\; n=p+1,\; \G_\kappa=PGL_2(p)$ or $\G_\kappa=PSL_2(p)$.
        
        The graph contains a subtree of order $p$. Because the coloring group corresponding to this tree has prime degree, it is primitive and so must fall under case (iii) of the theorem. If the coloring group corresponding to this tree were to contain the alternating group or be $M_{23}$, it's clear that $PGL_2(p)$ and $PSL_2(p)$ would be too small to contain it. It must then be the case that $PGL_d(q)\leq\G_\kappa\leq PGL_2(p)$ where $p=\frac{q^d-1}{q-1}$ for some $d\geq2$. This containment implies that the order of $PGL_d(q)$ divides the order of $PGL_2(p)$. That is: \begin{align*}
            \frac{(q^d-1)(q^d-1)\cdots(q^d-q^{d-1})}{q-1} \mid p^3-p.
        \end{align*} We proceed by showing this is impossible.

        By applying Cauchy's bound to the polynomial $q^{\binom{d-2}{2}-1}(q^d-q)-(q^d-1)^2$ and noting its positive leading coefficient, we obtain the inequality \begin{align}
            q^{\binom{d-2}{2}-1}>\frac{(q^d-1)^2}{q^d-q}\label{eq:inequality}
        \end{align} as long as $q\geq3,d\geq3$. From this inequality it is straightforward to derive that $|PGL_d(q)|>|PGL_2(p)|$:
        \begin{align*}
            \frac{(q^d-1)^2}{q^d-q}&<q^{\binom{d-2}{2}-1}\\
            &<q^{\binom{d-2}{2}-1}\prod_{i=2}^{d-1}(q^{d-1}-i)\\
            &=\prod_{i=2}^{d-1}(q^d-q^i)\\
            &<(q-1)^2\prod_{i=2}^{d-1}(q^d-q^i)\\
            \frac{(q^d-1)^2}{(q^d-q)^2}&<\prod_{i=1}^{d-1}(q^d-q^i)\\
            p^3&<p\prod_{i=1}^{d-1}(q^d-q^i)\\
            &=|PGL_d(q)|
        \end{align*}

        Hence, as long as $q\geq3$ and $d\geq3$, we have that $|PGL_d(q)|>p^3>p^3-p=|PGL_2(p)|$ and the above containment is impossible.

        For the case $d=2$, the containment implies that \[\frac{(q^2-1)(q^2-q)}{q-1}=p(q^2-q)\mid p^3-p=p(q-2)q.\] This implies that $q-1\mid q-2$, which is impossible for $q>1$.

        The final case is when $q=2$ and $d\geq 3$. If Inequality \ref{eq:inequality} does not hold, it is straightforward to show that $d\leq8$. The only values between $3$ and $8$ for which $2^d-1$ is prime are $d=3,5,$ or $7$. When $d=3$, we have that $p=7$, so $\G_\kappa$ must contain $S_7$ (See \Cref{sec:data}). For $d=5,7$ we have that $|PGL_d(q)|>|PGL_2(p)|$.

        \item $n-m=1$,\;$\G_\kappa=M_{11}$
        
        In this case, the graph contains a subtree of order 10. The orders of coloring groups generated on a tree of order 10 with at least three colors are $200,\; 240,\;  2^5\cdot 5!,\; 120^2,$ and $10!$ (see \Cref{sec:data}). An enumeration of the subgroups of $M_{11}$ shows that it has no subgroups of these orders.
        \item $n-m=1,\;\G_\kappa=M_{12}$

        The orders of coloring groups on a tree with 11 vertices with at least three colors are $\frac{11!}2$ or $11!$ (See \Cref{sec:data}). Both of these orders are too large to be contained in $M_{12}$.

        \item $n-m=0$, \;$\G_\kappa=PSL_2(11)$ or $\G_\kappa=M_{11}$.

        See the orders given above for coloring groups on a tree with 11 vertices.

        \item $n-m=0$, \;$n=p$, a prime,  $C_p\leq \G_\kappa\leq AGL_1(p)$

        Note that $|AGL_1(p)|=p(p-1)$, but \Cref{lem:size_bound} tells us that a tree on a prime number of vertices with at least three colors produces a coloring group of order at least $p^2$.
    \end{itemize}
\end{proof}

Although we have not ruled out case (iii)b in \Cref{thm:primitive_trees}, we have no example of a \emph{tree} $G$ which realizes $M_{23}$ as a coloring group, and we conjecture that no such graph exists. Note that the nonexistence of such a tree would also rule out case (ii)c.

\begin{lemma}\label{lem:centralizer_of_tree}
    For $(G,\kappa)$ a proper edge coloring of a tree, $\Aut_\kappa(G)$ is either trivial or has order two.
\end{lemma}

\begin{proof}

    If $G$ has a unique center, then every automorphism of $G$ must fix this center. By \Cref{lem:fixing_a_vertex}, $\Aut_\kappa(G)$ is trivial.

    Alternatively, if $G$ has two vertices $v_1$ and $v_2$ of maximum eccentricity, a nontrivial automorphism $\varphi$ must swap $v_1$ and $v_2$. If $\psi$ were another such automorphism, then $\psi\varphi$ would fix $v_1$ and $v_2$, and by \Cref{lem:fixing_a_vertex}, $\psi\varphi$ is the trivial automorphism and $\psi=\varphi^{-1}$. Hence, if $\Aut_\kappa(G)$ is nontrivial, it consists of a single automorphism of order 2.
\end{proof}

The following theorem characterizes coloring groups on trees by their centralizer using \Cref{thm:centralizer_of_coloring}.

\begin{theorem}\label{thm:signed_permutations}
    For $(G,\kappa)$ a proper edge coloring of a tree, either $C_{S_{V(G)}}(\G_\kappa)$ is trivial, or $|G|=2m$ is even and $\mathfrak{G}_\kappa$ is isomorphic to a subgroup of the type B coxeter group $B_m$.
\end{theorem}

\begin{proof} 
    If $\Aut_\kappa(G)$ is nontrivial and $e$ is the edge between the two vertices of $G$ of maximum eccentricity. The graph $G\setminus\{e\}$ is the disjoint union of two trees $T_1$ and $T_2$. The nontrivial element $\varphi\in\Aut_\kappa(G)$ restricts to an isomorphism between $T_1$ and $T_2$. Let $\kappa_1$ be the edge coloring of $T_1$ and let $\kappa'$ be the edge coloring of $G$ modified so that the edge $e$ has a unique color $c$. It's clear that $\G_{\kappa_1}$ is a subgroup of $\G_{\kappa'}$. Let \begin{align*}
        N&=\{g\tau_cg^{-1}:g\in\G_{\kappa'}\}\cup\{1\}\subseteq \G_{\kappa'}.
    \end{align*} Then $N$ is a normal subgroup of $\G_{\kappa'}$ isomorphic to ${C_2}^m$ where $m$ is the order of $T_1$ and $\G_{\kappa'}=N G_{\kappa_1}$. Hence, $\G_{\kappa'}=N\rtimes G_{\kappa_1}$ where $G_{\kappa_1}$ acts on $N$ by permutation, meaning $\G_{\kappa'}$ can be recognized as the wreath product \begin{align*}
        \G_{\kappa'}&\cong C_2\omega \G_{\kappa_1}\\
        &\leq C_2\omega S_m\\
        &\cong B_m.
    \end{align*}

    Finally, we have that the original coloring group $\G_\kappa\subseteq \G_{\kappa'}$.
\end{proof}

\begin{example}\label{ex:symmetric_tree} The following tree has a color-preserving automorphism that reflects the graph horizontally. The labels are taken from $\{-6,\dots,-1,1,\dots,6\}$ to show how the corresponding coloring group embeds in the group $B_6$ of signed permutations.
    
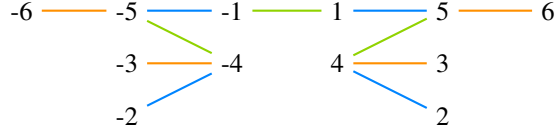
\begin{center}
         \begin{tikzpicture}[scale = .35]
        \node (1) at (-10, 0) {-6};
        \node (2) at (-6, 0) {-5};
        \node (3) at (-2, -2) {-4};
        \node (4) at (-6, -2) {-3};
        \node (5) at (-6,-4) {-2};
        \node (6) at (-2, 0) {-1};
        \node (7) at (2, 0) {1};
        \node (8) at (6, -4) {2};
        \node (9) at (6, -2) {3};
        \node (10) at (2,-2) {4};
        \node (11) at (6,0) {5};
        \node (12) at (10,0) {6};
        \draw[color = c1, thick] (2) -- (6);
        \draw[color = c1, thick] (3) -- (5);
        \draw[color = c1, thick] (8) -- (10);
        \draw[color = c1, thick] (7) -- (11);
        \draw[color = c2, thick] (3) -- (4);
        \draw[color = c2, thick] (9) -- (10);
        \draw[color = c2, thick] (1) -- (2);
        \draw[color = c2, thick] (11) -- (12);
        \draw[color = c3, thick] (2) -- (3);
        \draw[color = c3, thick] (10) -- (11);
        \draw[color = c3, thick] (6) -- (7);
        
    \end{tikzpicture}
    \end{center}
\end{example}

Finally, we provide a sufficient condition for a coloring group on a tree of order $n$ to be isomorphic to the symmetric group $S_n$.

\begin{definition} Let $G$ be a graph with proper edge coloring $\kappa:E(G)\to[k]$. Let $e$ be an edge such that all edges incident to $e$ are of distinct colors. We call $e$ a \emph{symmetric edge} if there exists a subset $S$ of $[k]$ containing the colors of all edges incident to $e$ but not containing the color of $e$ such that removing all edges colored by a color in $S$ leaves a graph with exactly one even component. 
\end{definition}

\begin{example}

The following edge coloring has a symmetric edge labeled $e$. Notice that removing the edges colored blue and orange leave only a single connected component with even order.
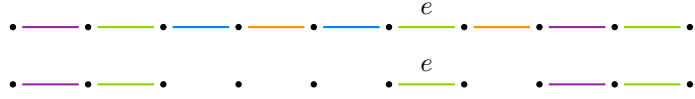
\begin{center}
    \begin{tikzpicture}[scale = .5]
        \node (Label) at (-11, 0) {};
        \node (Label) at (2, .5) {$e$};
        \node (1) at (-9, 0) {};
        \node (2) at (-7, 0) {};
        \node (3) at (-5, 0) {};
        \node (4) at (-3, 0) {};
        \node (5) at (-1,0) {};
        \node (6) at (1, 0) {};
        \node (7) at (3, 0) {};
        \node (8) at (5, 0) {};
        \node (9) at (7, 0) {};
        \node (10) at (9,0) {};
        \draw[color = c4, thick] (1) -- (2);
        \draw[color = c3, thick] (3) -- (2);
        \draw[color = c1, thick] (3) -- (4);
        \draw[color = c2, thick] (5) -- (4);
        \draw[color = c1, thick] (5) -- (6);
        \draw[color = c3, thick] (6) -- (7);
        \draw[color = c2, thick] (7) -- (8);
        \draw[color = c4, thick] (8) -- (9);
        \draw[color = c3, thick] (9) -- (10);
        \draw[color = black, fill] (1) circle (2pt);
        \draw[color = black, fill] (2) circle (2pt);
        \draw[color = black, fill] (3) circle (2pt);
        \draw[color = black, fill] (4) circle (2pt);
        \draw[color = black, fill] (5) circle (2pt);
        \draw[color = black, fill] (6) circle (2pt);
        \draw[color = black, fill] (7) circle (2pt);
        \draw[color = black, fill] (8) circle (2pt);
        \draw[color = black, fill] (9) circle (2pt);
        \draw[color = black, fill] (10) circle (2pt);
    \end{tikzpicture}
\end{center}

\begin{center}
    \begin{tikzpicture}[scale = .5]
        \node (Label) at (-11, 0) {};
        \node (Label) at (2, .5) {$e$};
        \node (1) at (-9, 0) {};
        \node (2) at (-7, 0) {};
        \node (3) at (-5, 0) {};
        \node (4) at (-3, 0) {};
        \node (5) at (-1,0) {};
        \node (6) at (1, 0) {};
        \node (7) at (3, 0) {};
        \node (8) at (5, 0) {};
        \node (9) at (7, 0) {};
        \node (10) at (9,0) {};
        \draw[color = c4, thick] (1) -- (2);
        \draw[color = c3, thick] (3) -- (2);
        \draw[color = c3, thick] (6) -- (7);
        \draw[color = c4, thick] (8) -- (9);
        \draw[color = c3, thick] (9) -- (10);
        \draw[color = black, fill] (1) circle (2pt);
        \draw[color = black, fill] (2) circle (2pt);
        \draw[color = black, fill] (3) circle (2pt);
        \draw[color = black, fill] (4) circle (2pt);
        \draw[color = black, fill] (5) circle (2pt);
        \draw[color = black, fill] (6) circle (2pt);
        \draw[color = black, fill] (7) circle (2pt);
        \draw[color = black, fill] (8) circle (2pt);
        \draw[color = black, fill] (9) circle (2pt);
        \draw[color = black, fill] (10) circle (2pt);
    \end{tikzpicture}
\end{center}

An edge coloring can also have symmetric edge incident to a leaf. Notice that removing the edges colored orange as well as those colored blue in the following graph coloring produces a graph with only a single connected component with even order.

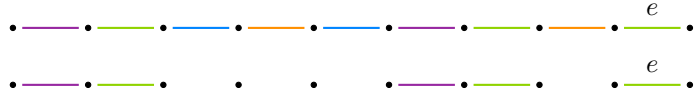
\begin{center}
    \begin{tikzpicture}[scale = .5]
        \node (Label) at (-11, 0) {};
        \node (Label) at (8, .5) {$e$};
        \node (1) at (-9, 0) {};
        \node (2) at (-7, 0) {};
        \node (3) at (-5, 0) {};
        \node (4) at (-3, 0) {};
        \node (5) at (-1,0) {};
        \node (6) at (1, 0) {};
        \node (7) at (3, 0) {};
        \node (8) at (5, 0) {};
        \node (9) at (7, 0) {};
        \node (10) at (9,0) {};
        \draw[color = c4, thick] (1) -- (2);
        \draw[color = c3, thick] (3) -- (2);
        \draw[color = c1, thick] (3) -- (4);
        \draw[color = c2, thick] (5) -- (4);
        \draw[color = c1, thick] (5) -- (6);
        \draw[color = c4, thick] (6) -- (7);
        \draw[color = c3, thick] (7) -- (8);
        \draw[color = c2, thick] (8) -- (9);
        \draw[color = c3, thick] (9) -- (10);
        \draw[color = black, fill] (1) circle (2pt);
        \draw[color = black, fill] (2) circle (2pt);
        \draw[color = black, fill] (3) circle (2pt);
        \draw[color = black, fill] (4) circle (2pt);
        \draw[color = black, fill] (5) circle (2pt);
        \draw[color = black, fill] (6) circle (2pt);
        \draw[color = black, fill] (7) circle (2pt);
        \draw[color = black, fill] (8) circle (2pt);
        \draw[color = black, fill] (9) circle (2pt);
        \draw[color = black, fill] (10) circle (2pt);
    \end{tikzpicture}
\end{center}
\begin{center}
    \begin{tikzpicture}[scale = .5]
        \node (Label) at (-11, 0) {};
        \node (Label) at (8, .5) {$e$};
        \node (1) at (-9, 0) {};
        \node (2) at (-7, 0) {};
        \node (3) at (-5, 0) {};
        \node (4) at (-3, 0) {};
        \node (5) at (-1,0) {};
        \node (6) at (1, 0) {};
        \node (7) at (3, 0) {};
        \node (8) at (5, 0) {};
        \node (9) at (7, 0) {};
        \node (10) at (9,0) {};
        \draw[color = c4, thick] (1) -- (2);
        \draw[color = c3, thick] (3) -- (2);
        \draw[color = c4, thick] (6) -- (7);
        \draw[color = c3, thick] (7) -- (8);
        \draw[color = c3, thick] (9) -- (10);
        \draw[color = black, fill] (1) circle (2pt);
        \draw[color = black, fill] (2) circle (2pt);
        \draw[color = black, fill] (3) circle (2pt);
        \draw[color = black, fill] (4) circle (2pt);
        \draw[color = black, fill] (5) circle (2pt);
        \draw[color = black, fill] (6) circle (2pt);
        \draw[color = black, fill] (7) circle (2pt);
        \draw[color = black, fill] (8) circle (2pt);
        \draw[color = black, fill] (9) circle (2pt);
        \draw[color = black, fill] (10) circle (2pt);
    \end{tikzpicture}
\end{center}
\end{example}

\begin{theorem}\label{thm:symmetric_edge}
    Let $(G,\kappa)$ be a proper edge coloring of a tree. If $G$ contains a symmetric edge, then $\G_\kappa\cong S_{V(G)}$.
\end{theorem}

\begin{proof} Let $e=\{i,j\}$ be a symmetric edge in $G$ with corresponding subset of colors $S$. Let $H$ be the graph obtained by deleting all edges colored by a color in $S$ and $m$ the least common multiple of the orders of odd-sized connected components in $H$. Let \begin{align*}
    \pi=\prod_{a\in [k]\setminus S}\tau_a,
\end{align*} with the product taken in any order. Then because the edge $e$ is the only even-sized component in $\sigma$, \Cref{lem:long_cycle} says that \begin{align*}
    \pi^m=(i,j).
\end{align*}

Let $K_i$ and $K_j$ be the sets of colors of edges incident to $i$ and $j$ respectively. Let \begin{align*}
    \sigma_1&=\prod_{a\in K_i\setminus\{\kappa(e)\}}\tau_a\\
    \sigma_2&=\prod_{a\in [k]\setminus K_i} \tau_a
\end{align*} with products taken in any order. It's clear that $\sigma_2(i)=i$, but because $K_i\cap K_j=\{\kappa(e)\}$, we also have that $\sigma_1(j)=j$. By \Cref{lem:long_cycle}, $\sigma=\sigma_1\tau_{\kappa(e)}\sigma_2$ is an $n$-cycle, and by the above observation, $\sigma(i)=\sigma_1\tau_{\kappa(e)}(i)=\sigma_1(j)=j$. The transposition $(i,j)$ and the $n$-cycle $\sigma$ with $\sigma(i)=\sigma(j)$ generate the full symmetric group $S_{V(G)}$.
\end{proof}

The symmetric edge condition has allowed us to more efficiently search colorings of trees by weeding out colorings which we know will give the symmetric group. See \Cref{sec:data} for some data on coloring groups on small trees.

\section{Applications}\label{sec:applications}

\subsection{Generalized Toggle Groups}

The following condition for a coloring group on a tree to be isomorphic to a generalized toggle group is due to Jonathan Bloom. 

\begin{definition}
    A word $w\in[k]^\ell$ is called a \emph{toggle word} if every contiguous subword of $w$ has at least two distinct letters that occur with odd parity.
\end{definition}

\begin{theorem}
    Let $(G,\kappa)$ be a proper edge coloring of a tree. A necessary condition for $\G_\kappa$ to be a generalized toggle group is that for every path of length greater than one, reading the edge colors along that path must be a toggle word.
\end{theorem} \begin{proof}
     This follows from the following observation on toggle graphs. Let $G = P_L$ for some toggle group $T(L)$ and suppose that $w_1 w_2 \dots w_\ell$ is the sequence of elements of the base set corresponding to the coloring for a walk in $G$ with no consecutive letters. Without loss of generality we can assume no vertices of $G$ are repeated.
    
    If each distinct label appears an even number of times, then each element element is toggled in and out an equal number of times so this is a nontrivial closed walk in the graph without consecutive repeated edges, so the walk must be a cycle. 

    If exactly one label appears an odd number of times, call it $e$, then the result of this walk is that $e$ is either toggled in or out of the initial set $S$, as all other labels are toggled in and out an even number of times. That is, this path ends at $\tau_e(S)$. Because $G$ arose from a toggle group, there must be an edge labeled $e$ between $S$ and $\tau_e(S)$. This single edge constitutes a path distinct from $w_1w_2\dots w_\ell$ between $S$ and $\tau_e(S)$, so $G$ must contain a cycle.


    As such if $G$ is acyclic and $\kappa$ corresponds to a coloring which arises from a toggle group, then no nontrivial path of distinct vertices in $G$ can have fewer than two colors which appear an odd number of times.
\end{proof}

\begin{example} The following edge coloring gives rise to a coloring group isomorphic to a generalized toggle group. The subsets are the vertex labels.

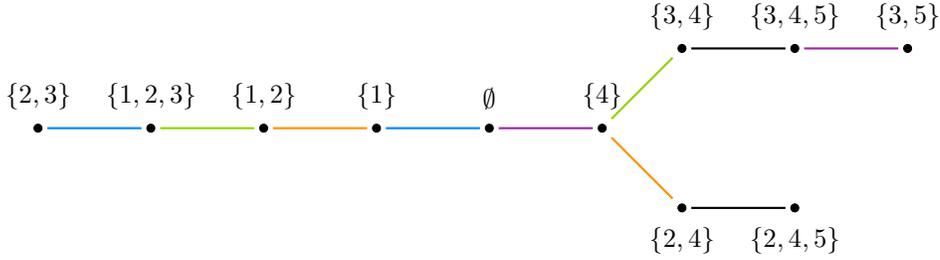
\newcommand{\com}{,}
    \begin{center}
        \begin{tikzpicture}[scale = .75]
        \node[label=$\{2\com3\}$] (1) at (-9, 0) {};
        \node[label=$\{1\com2\com3\}$] (2) at (-7, 0) {};
        \node[label=$\{1\com2\}$] (3) at (-5, 0) {};
        \node[label=$\{1\}$] (4) at (-3, 0) {};
        \node[label=$\emptyset$] (5) at (-1,0) {};
        \node[label=$\{4\}$] (6) at (1, 0) {};
        \node[label=$\{3\com4\}$] (7) at (1+2/1.414, 2/1.414) {};
        \node[label=$\{3\com4\com5\}$] (8) at (3+2/1.414, 2/1.414) {};
        \node[label=$\{3\com5\}$] (9) at (5+2/1.414, 2/1.414) {};
        \node[label=below:$\{2\com4\}$] (10) at (1+2/1.414, -2/1.414) {};
        \node[label=below:$\{2\com4\com5\}$] (11) at (3+2/1.414, -2/1.414) {};
        \draw[color = c1, thick] (1) -- (2);
        \draw[color = c3, thick] (2) -- (3);
        \draw[color = c2, thick] (3) -- (4);
        \draw[color = c1, thick] (5) -- (4);
        \draw[color = c4, thick] (5) -- (6);
        \draw[color = c3, thick] (6) -- (7);
        \draw[color = black, thick] (7) -- (8);
        \draw[color = c4, thick] (8) -- (9);
        \draw[color = c2, thick] (6) -- (10);
        \draw[color = black, thick] (10) -- (11);
        \draw[color = black, fill] (1) circle (2pt);
        \draw[color = black, fill] (2) circle (2pt);
        \draw[color = black, fill] (3) circle (2pt);
        \draw[color = black, fill] (4) circle (2pt);
        \draw[color = black, fill] (5) circle (2pt);
        \draw[color = black, fill] (6) circle (2pt);
        \draw[color = black, fill] (7) circle (2pt);
        \draw[color = black, fill] (8) circle (2pt);
        \draw[color = black, fill] (9) circle (2pt);
        \draw[color = black, fill] (10) circle (2pt);
        \draw[color = black, fill] (11) circle (2pt);
    \end{tikzpicture}
    \end{center} The path from the vertex labeled $\{2,3\}$ to the vertex labeled $\{3,5\}$ corresponds to the toggle word \[w=13214354.\]
\end{example}

Note that in \Cref{ex:imprimitive_path}, any path starting in the left-most block of vertices (colored blue) and ending at the next vertex it encounters in that left-most block must encounter every edge color in-between an even number of times. The colors on this path will not be a toggle word, so the coloring group in this example cannot be a generalized toggle group. A similar insight applies generally to trees as stated in the next theorem.

\begin{theorem}\label{thm:toggleTreePrim} Let $(G,\kappa)$ be a proper edge coloring of a tree $G$ with order $n\geq3$. If $\G_\kappa$ is a generalized toggle group, then $\G_\kappa$ is primitive.
\end{theorem} \begin{proof}
    Suppose $\G_\kappa$ were not primitive. Then there would be at least two blocks in a system of imprimitivity. Because $G$ is connected, there must be some edge colored $a$ connecting two vertices $i$ and $j$ in the same block. Furthermore, there must be an edge colored $b$ connecting $i$ to a vertex $i'$ in another block. By the definition of an imprimitive coloring, there must then be an edge also colored $b$ connecting $j$ to some $j'$ in the same block as $i'$. Then the path $i',i,j,j'$ has word $aba$, which is not a toggle word, contradicting the assumption that $\G_\kappa$ is a generalized toggle group.
\end{proof}

This result also follows from a more general result of Bloom and Saracino~\cite[Theorem 4.2]{bloom2023transitive}, specifically that an imprimitive toggle group $T(L)$ containing a cycle on every element of $L$ will have that $L$ has the structure of a product of sets, which must contain a cycle. Additionally Theorem~\ref{thm:toggleTreePrim} allows us to apply \Cref{thm:primitive_trees}(iii) to the particular case of a generalized toggle group presented by an edge coloring of a path.

\begin{theorem} Let $(G,\kappa)$ be a proper edge coloring of a path. If $\G_\kappa$ is a generalized toggle group, then one of the following statements holds:
    \begin{enumerate}
        \item[(i)] $n=\frac{q^d-1}{q-1}$ for $q$ a prime power, $d\geq2$, and $PGL(d,q)\subseteq \G_\kappa\subseteq P\Gamma L(d,q)$.
        \item[(ii)] $\G_\kappa\cong M_{23}$, the Mathieu group. 
        \item[(iii)] $\G_\kappa\cong S_n$.
    \end{enumerate}
\end{theorem}

\subsection{Independence Posets}
When Cameron and Fon-der-Flaass originally described the toggle group, they were essentially considering it as a group structure on a distributive lattice. While the generalized toggle group is an extension of these order ideal toggles from a group-theoretic perspective to subsets of a finite set, there has historically been considerably more focus on trying to understand generalizations of this family of lattices from a combinatorial as well as order-theoretic point of view. One particular family that serves as a generalization is that of trim lattices, first introduced by Thomas in~\cite{thomas2006analogue}. Before we define these lattices, recall that an element $x$ of a lattice $L$ is \emph{left modular} if for any $y < z \in L$ the following equality holds $y \vee (x \wedge z)= (y\vee x)\wedge z$. 
\begin{definition}[{\cite{thomas2006analogue}}]
    A lattice is \emph{trim} if it has a maximal chain of $(n+1)$ left modular elements as well as $n$ join-irreducible elements and $n$ meet irreducible elements.
\end{definition} 
These lattices were further generalized by Thomas with Williams to a family of posets known as independence posets, which contains many posets of combinatorial and geometric interest. Importantly  the set of independence posets which are lattices is precisely the set of trim lattices. As such all of the following are examples of independence posets, as explained in~\cite{thomas2019Independence}. 
\begin{itemize}
    \item Distributive lattices
    \item Tamari Lattices
    \item Cambrian Lattices
    \item Fuss-Cambrian Lattices
\end{itemize}
as well as many more.
To be self-contained, we follow the treatment of ~\cite{thomas2019Independence} to define these objects, up to the omission of proofs.

For the remainder of the background on independence posets, let $G$ be a directed acyclic graph. The associated $G$-order is the partial order on $V(G)$ where $g_1\ge g_2$ when there is a directed path from $g_1$ to $g_2$. Recall that an independent set of $G$ is a subset where no two elements are adjacent. 
To define an independence poset we first need the following definitions. 

\begin{definition}[{\cite[Definition 1.2]{thomas2019Independence}}]
    A pair ($D$, $U$) of independent sets of $G$ is called \emph{orthogonal} if there is no edge in $G$ from an element of $D$ to an element of $U$. An orthogonal pair of independent sets ($D$, $U$) is called \emph{tight} if whenever any element of $D$ is increased (removed and replaced by a larger element with respect to the $G$-order), any element of $U$ is decreased, or a new element is added to either $D$ or $U$, then the result is no longer an orthogonal pair of independent sets. We abbreviate tight orthogonal pair by $\top$, and we write $\top(G)$ for the set of all tops of $G$.
\end{definition}

It turns out that when $G$ is directed and acyclic then specifying either of $U$ or $D$ gives a bijection to the independent sets of $G$. This result is more formally stated as follows. 

\begin{theorem}[{\cite[Theorem 1.2]{thomas2019Independence}}]
    Let $I$ be an independent set of a directed acyclic graph $G$. Then there exists a unique $(I, U) \in \top(G)$ and a unique $(D, I) \in \top(G)$.
\end{theorem}

To state the last required definition, let $\ell$ be a fixed linear extension of the $G$-order and let $\ell'$ be a fixed dual linear extension of the $G$-order. 

\begin{definition}[{\cite[Definition 1.3]{thomas2019Independence}}]
    The \emph{flip} of $(D, U) \in \top(G)$ at an element $g \in G$ is the tight orthogonal pair $flip_g(D, U)$ defined as follows: if $g\notin D$ and $g \notin U$, the flip does nothing. Otherwise, preserve all elements of $D$ that are not less than $g$ and all elements of $U$ that are not greater than $g$ (and delete all other elements); after switching the set to which $g$ belongs, then greedily complete $D$ and $U$ to a tight orthogonal pair in the orders $\ell'$ and $\ell$, respectively.
\end{definition}

These flips are all well defined, and in fact involutions, as shown in~\cite[Lemma 3.2]{thomas2019Independence}. One additional consequence of~\cite[Lemma 3.2]{thomas2019Independence} that we mention now is that if $g,h$ are incomparable elements of a $G$-order, then the associated flips commute. All together an independence poset is defined as follows. 

\begin{definition}[{\cite[Definition 1.4]{thomas2019Independence}}]
    The \emph{independence poset} on $\top(G)$ is the reflexive and transitive closure of the relations $(D, U)\lessdot (D', U')$ if there is some $g \in U$ such that $flip_g (D, U) = (D', U'). $
\end{definition}

For full details that this is in fact well defined as a poset, see~\cite[Lemma 3.3]{thomas2019Independence}. The flips then naturally label the edges in the Hasse diagram of $\top(G)$ by elements of $G$. Importantly for us, this is clearly a proper edge labeling of the Hasse diagram, so one can then ask about the coloring group generated by these labels. We can fully describe these groups with the following result. 
\begin{example}
    \label{fig:ind_pos}
     The independence poset on the increasing orientation of $P_3$
     
    \centering
\begin{tikzpicture}
    \node (1) at (-2, .5) {1};
    \node (2) at (0, .5) {2};
    \node (3) at (2, .5) {3};
    \draw[->, thick] (1) -- (2);
    \draw[->, thick] (2) -- (3);
    \node (e) at (6,-1) {$(\emptyset, 13)$};
    \node (3') at (4,-.25) {$(3, 2)$};
    \node (2') at (4,1) {$(2,1)$};
    \node (1') at (8,.375) {$(1, 3)$};
    \node (13') at (6,2) {$(13,\emptyset)$};
    \draw [-] (e) --node[below right]{1} (1');
    \draw [-] (e) --node[below left]{3} (3');
    \draw [-] (1') --node[above]{3} (13');
    \draw [-] (3') --node[left]{2} (2');
    \draw [-] (2') --node[above left]{1} (13');
\end{tikzpicture}
\end{example}

\begin{theorem}\label{thm:IndPosStruct}
    Let $P=\top(G)$ be an independence poset with $|P|=n$.
    
    \begin{itemize}
        \item If $G$ is connected, then the coloring group generated by the labeling of the Hasse diagram of $P$ contains $A_n$.
        \item If $G$ is not connected, then the coloring group is isomorphic to the direct product of the coloring groups restricted to the connected components.
    \end{itemize}
\end{theorem}
 The first statement requires nontrivial machinery which we now explain. 
 
 To prove this, we first recall the following definitions due to Striker~\cite{striker2018rowmotion}. The first is that a toggle group is said to be \emph{toggle-alternating} if it contains the alternating group. The second and more involved sequence of definitions are those needed for a collection to be \emph{inductively toggle-alternating}.

\begin{definition}[{\cite[Definition 2.20]{striker2018rowmotion}}]
Given $e\in E$ and $L\subseteq 2^E$, define $L_e:=\{X\in L \ | \ e\in X\}$ to consist of all the subsets in $L$ that contain $e$ and $L_{\thickbar{e}}:=\{X\in L \ | \ e\notin X\}$  to be all the subsets in $L$ that do not contain $e$. For any subset $S\subseteq L$, let $\tau_e(S)=\left\{\tau_e(X) \mid X\in S\right\}$.

\end{definition}
\begin{definition}[{\cite[Definition 2.5]{striker2018rowmotion}}]
Given $L\subseteq 2^E$, let an \emph{essentialization} of $(L,E)$ be $(L',E')$ where $E'$ and $L'$ are as follows. Every $e\in E'$ must be an element $e\in E$ such that there exist $X,Y\in L$ with $e\in X$ and $e\not\in Y$. That is, no element of $E'$ may be contained in all or none of the subsets in $L$. Also, consider subsets $Z\subseteq E$ with $|Z|>1$ such that if $e\in Z$ and $e\in W\in L$, then $Z\subseteq W$. For any maximal such $Z$, we exclude all elements of $Z$ except one from $E'$. Let $L'=\{X\cap E' \ | \ X\in L\}$. 
\end{definition}
\begin{definition}[{\cite[Definition 2.21]{striker2018rowmotion}}]
\label{def:pleasant}
Define the collection of \emph{inductively toggle-alternating} sets $L$ to be the smallest collection of $L\subseteq 2^E$, with essentializations $(L',E')$, such that:
\begin{enumerate}
\item 
If $|E'|\leq 4$, then $T(L)$ is toggle-alternating.
\item If $|E'|\geq 5$, 
then there exists $e\in E'$ such that at least one of the following holds:
\begin{enumerate}
\item
$L_e$ is inductively toggle-alternating, $L_e\cup \tau_e(L_e) = L$, and $L_e\cap \tau_e(L_e) \neq \emptyset$, or
\item $L_{\thickbar{e}}$ is inductively toggle-alternating, $L_{\thickbar{e}}\cup \tau_e(L_{\thickbar{e}}) = L$, and $L_{\thickbar{e}}\cap \tau_e(L_{\thickbar{e}}) \neq \emptyset$.
\end{enumerate}
\end{enumerate}
\end{definition}

This definition allows for the statement of the following result, which we will generalize.

\begin{theorem}[{\cite[Theorem 2.22]{striker2018rowmotion}}]
\label{thm:toggpstructure}
If $L$ is inductively toggle-alternating, then it is toggle-alternating.
\end{theorem}

Roughly, we will generalize the above definitions and the hypothesis of Theorem~\ref{thm:toggpstructure} from the setting of toggle groups to that of coloring groups. In the style of~\cite{striker2018rowmotion}, we now refer to a coloring group $\G_\kappa$, or equivalently the pair $(G,\kappa)$, as being \emph{color-alternating} if $\G_\kappa$ contains the alternating group on the vertex set of $G$. This then allows us to describe a straightforward generalization of inductively toggle-alternating. 

\begin{definition}\label{def:ind_color_alt}
    Let $F$ be a collection of connected graphs with proper surjective edge colorings such that the following hold. 
    \begin{enumerate}
        \item If $(G,\kappa)\in F$ with $|\kappa(E(G))| \le 4$ the group $\G_\kappa$ is color-alternating.
        \item If $(G,\kappa)\in F$ with $|\kappa(E(G))|\ge 5$ there is some color $i$ such that the vertices of $G$ can be partitioned into two disjoint subsets $G_1,G_2$ (thought of as induced subgraphs) such that \begin{enumerate}
            \item $\tau_i(G_1)\cap G_1 \neq \emptyset$,
            \item $\tau_i(G_1)\cup G_1 = G$, $(G_1,\kappa|_{G_1})\in F$,
            \item the only edges that connect $G_1$ and $G_2$ are colored $i$.
        \end{enumerate}
    \end{enumerate}

    We call the elements of $F$ inductively color-alternating.
\end{definition}
\begin{theorem}\label{thm:inductive_color_alternating}
    If $(G,\kappa)$ is inductively color-alternating, then $\G_\kappa$ is color-alternating.
\end{theorem}
\begin{proof}
    The proof is almost exactly the proof of \ref{thm:toggpstructure} due to Striker in~\cite{striker2018rowmotion}, which is essentially the proof in the case of order ideals of Cameron and Fon-der-Flaass~\cite{cameron1995orbits}. 

    For the remainder of the proof, $C$ will denote the codomain of $\kappa$ (with $C'$ the associated codomain of $\kappa'$). After relabeling we will identify $C$ with $\{1,\dots,|C|\}$. 

    Note for $|C|\leq 4$, we have required in the definition of inductively color-alternating that $\G_\kappa$ be color-alternating (this is because in the proof below we need $A_{|V(G)|}$ to be simple, and the alternating group ${A}_n$ is simple for $n\geq 5$ but not for $n=4$), so the case that $|C|\leq4$ serves as one of our base cases. The other, which we call the exceptional base case, is when $|V(G)|= 4$ and $|C|>4$. In this exceptional base case, $G$ is either $K_4$ or $K_4$ with an edge removed. If $G= K_4$ minus an edge, and $|C|=5$ then $\G_\kappa$ is the symmetric group as every edge has a distinct color. If $G = K_4$ and $|C|=6$, then $\G_\kappa$ is the symmetric group for the same reason. If $G = K_4$ and $|C|=5$, then without loss of generality we can assume the edges $(1,2)$ and $(3,4)$ have the same color, as exactly one color is repeated. If we just ignore these two edges, the remaining edges form a connected graph on four vertices with all distinct edge colors, so the corresponding group is the symmetric group. We have shown that our base cases are all color-alternating.

    Now suppose $|C|\geq 5$. Also, suppose for all $(G',\kappa')\in F$ with $|C'|<|C|$ we have $(G',\kappa')$ is color-alternating. Let $i\in C$ be the color required by condition $(2)$ of Definition~\ref{def:ind_color_alt}, since $(G,\kappa)$ is inductively color-alternating. Following from the definition of inductively color-alternating, $G = G_1\cup G_2$, $|G_1| > |G_2| > 0$ as $G_1 \cup\tau_i(G_1) = G$ and $G_1\cap \tau_i(G_1)\neq \emptyset$ so $G_2$ is a proper subset of $\tau(G_1)$, the coloring pair obtained by considering $G_1$ as an induced subgraph of $G$ with the inherited proper edge coloring $(G_1, \kappa|_{G_1})= (G_1,\kappa')$ is inductively color-alternating, and $i\notin \kappa'(E(G_1))$. Importantly $\tau_i$ acts trivially on $G_1$ when the vertex set is restricted. Additionally we are here using the hypothesis that the only edges that connect $G_1$ and $G_2$ are colored $i$, so that the restricted action of the coloring of $\kappa'$ is the same as that of $\kappa$.  By induction, $\G_{\kappa'}$ contains the alternating group on $|G_1|$ elements. We first assume that $|G_1| \geq 5$, so $A_{|G_1|}$ is simple. Since $A_{G_1}$ is simple and $|G_1|>|G_2|$, the subgroup $K$ of $\G_\kappa$ fixing $G_2$ pointwise induces at least $A_{G_1}$ on $G_1$. If instead $|G_1| \le 4$ then $|C'| \le 4$, and we satisfy condition $(1)$ of Definition~\ref{def:ind_color_alt} or we are in an exceptional base case as no simple graph $G'$ on with $|G'|=4$ can have more than 6 edges and the cases where $|G'|=4$ and $|C|>4$ are the exceptional base cases. In either event the subgroup $K$ of $\G_\kappa$ fixing $G_2$ pointwise induces at least $A_{G_1}$ on $G_1$. 

Then in either case, $G_1\cup \tau_i(G_1) = G$ and $\tau_i(G_1)\cap G_1 \neq \emptyset$, so $K$ and the conjugate subgroup $K^{\tau_i}:=\left\{\tau_i k \tau_i \ | \ k\in K\right\}$ generate an alternating or symmetric group on~$G$. 
\end{proof}

Before proving Theorem~\ref{thm:IndPosStruct}, we mention a few more necessary results on the structure of independence posets from~\cite{thomas2019Independence}. The first is the \emph{tight orthogonal pair recursion}. For any $g\in G$, $\{g\}$ is an independent set, and there are unique $D, U$ such that $(D,\{g\})= m_g$ and $(\{g\},U)=j_g$ are tight orthogonal pairs. Let $\top_g(G) = [\hat{0},m_g]$ and $\top^g(G)=[j_g,\hat{1}]$. In the case where $g$ is an extremal (minimal or maximal) element of the $G$-order, these two intervals in fact partition the independence poset. Formally this is stated as follows.

\begin{lemma}[{\cite[Lemma 3.7]{thomas2019Independence}}]\label{lem:extreme_top_containment}
    Let $G$ be an acyclic directed graph. If $g$ is an extremal element of $G$ then $\top(G) = \top_g(G)\cup \top^g(G)$. Furthermore, 
    \begin{itemize}
        \item If $g$ is minimal, $(D,U)\in \top_g(G)$ if and only if $g\in U$.
        \item If $g$ is maximal, $(D,U)\in \top^g(G)$ if and only if $g\in D$.
    \end{itemize}
    In particular, if $x\in \top^g(G)$ and $y\in \top_g(G)$ then $x\not\leq y$.
\end{lemma}
While we do not repeat the proof of Lemma~\ref{lem:extreme_top_containment} we do note that one essential feature of the argument shows that when $g$ is an extremal element, for any $g'\neq g$, both $flip_{g'}(\top_g(G))=\top_g(G)$ and $flip_{g'}(\top^g(G)) = \top^g(G)$. In particular this provides potential candidates for $i$, $G_1$, and $G_2$ to apply Theorem~\ref{thm:inductive_color_alternating}. What then remains is to understand the action of $flip_g$ when $g$ is extremal. To that end we use the following results of~\cite{thomas2019Independence}.

\begin{lemma}[{\cite[Lemma 3.9]{thomas2019Independence}}]
Let $G$ be a directed acyclic graph.
\begin{itemize}
\item If $g$ is minimal and $(D,U) \in \top_g(G)$, then $flip_g(D,U)=(D\cup \{g\},U')$ for some $U'$.
\item If $g$ is maximal and $(D,U) \in \top^g(G)$, then $flip_g(D,U)=(D',U\cup\{g\})$ for some $D'$.
\end{itemize}
\label{lem:down}
\end{lemma}

Before stating the last result, we mention the following notation. If $G$ is a directed acyclic graph and $g\in G$, $G_g$ is $G\setminus \{g\}$ and $G_g^\circ$ is $G\setminus(\{g\}\cup\{g'|(g,g') \text{ or } (g',g)\in E(G)\}) $.

\begin{theorem}[{\cite[Theorem 3.10]{thomas2019Independence}}]
\label{thm:decomposition}
Let $g$ be an extremal element of an acyclic directed graph $G$.  Then \begin{align*} (D,U)&\mapsto (D ,U \setminus \{g\}\} \text{ is a bijection } \begin{cases} \top_g(G)\simeq\top(G_g^\circ) & \text{if } g \text{ minimal} \\ \top_g(G)\simeq\top(G_g) & \text{if } g \text{ maximal} \end{cases} \\
  (D,U) &\mapsto (D \setminus \{g\},U\}) \text{ is a bijection } \begin{cases} \top^g(G)\simeq\top(G_g) & \text{if } g \text{ minimal} \\ \top^g(G)\simeq\top(G_g^\circ) & \text{if } g \text{ maximal} \end{cases}\end{align*}
\end{theorem}

This recursive structure description is the final tool we need.

\begin{proof}[of Theorem~\ref{thm:IndPosStruct}]
    We note that the element $\tau_i$ induced by the proper edge coloring of the Hasse diagram of $\top(G)$ coming from the flips is exactly $flip_i$ so we will continue to use $flip_i$.
    To show the first statement, notice that if $G$ is connected, so is the Hasse diagram of the $G$-order. Particularly, the Hasse diagram of the $G$-order is the subgraph of $G$ obtained by deleting all edges $u,v$ if there is a longer path from $u$ to $v$. 
    
    Recall that for a finite connected poset $P$, there is always some extremal element $g$ for which the poset obtained by removing $g$ is connected. Importantly if $P$ is a $G$-order then the poset $P'$ obtained by the deletion of this extremal element $g$ is a $G_g$-order, so $G_g$ is still connected as a graph. 

    Now to show that the coloring groups of independence posets with connected directed acyclic graphs are inductively color-alternating, we show the inductive step, so we assume $|G| > 4$. The base cases where $|G| \le 4$ were checked by computer and the code is included in the repository linked in \Cref{sec:data}. 
    
    First choose an extremal element $g$ such that the Hasse diagram of the $G$-order is connected. We can assume $g$ is maximal, as otherwise we can swap the role of $\top_g(G)$ and $\top^g(G)$, so we have $G_1 = \top_g(G)$ and $G_2 = \top^g(G)$. Note that $|G_1|> |G_2|$ as $G_1$ is in bijection with independent sets of $G$ that do not contain $g$, whereas $G_2$ is in bijection with independent sets of $G$ that do not contain $g$ or a neighbor of $g$.
    
    By Lemma~\ref{lem:extreme_top_containment} we know that $G_1$ and $G_2$ partition the Hasse diagram of $\top(G)$.  Additionally by Lemma~\ref{lem:down} and Lemma~\ref{lem:extreme_top_containment} we have that $flip_g(G_2)\subseteq G_1$, so $G_1 \cap flip_g(G_1) \neq \emptyset$ as $|G_1|>|G_2|$. Furthermore by the discussion preceding Lemma~\ref{lem:down} we know that the only edges connecting $G_1$ and $G_2$ are colored by $g$. Finally note that by the choice of $g$ and Theorem~\ref{thm:decomposition}, $G_1 = \top_g(G)\simeq\top(G_g)$ which is an independence poset of a connected directed acyclic graph with $|G_g| < |G|$. As such the family of independence posets from connected directed acyclic graphs with coloring given by the labels of the flips is inductively color-alternating, and thus color alternating. 

    To prove the second part, note that no two elements of a connected component of $G$ are comparable in the $G$-order, so their associated flips commute. As such the coloring group is immediately seen to be isomorphic to the direct product of the groups of the connected components.
\end{proof}

\appendix

\section{Constructions of Coloring Groups on Trees}\label{sec:constructions}

In this appendix, we present some particular constructions of $S_n,A_n,D_n,$ and (when $n$ is even,) $B_{n/2}$ as coloring groups on trees of size $n$.

\begin{example} For a connected graph of order $n$, let $(G,\kappa)$ be the edge coloring that gives each edge a distinct color. Then it's clear that $\G_\kappa=S_n$.
\end{example}

On the other extreme, it's possible to construct $S_n$ with very few colors.

\begin{example}
    Given any proper edge coloring $(T,\kappa)$ of a tree $T$ of order $n$, let $\kappa'$ be the coloring obtained by changing the color of an edge incident to a leaf to a new unique color. The symmetric edge condition in \Cref{thm:symmetric_edge} guarantees that $\G_{\kappa'}=S_n$. In this way we can realize the symmetric group $S_n$ on any tree of order $n$ with just $\Delta(T)+1$ colors, where $\Delta(T)$ is the maximum degree of $T$. In particular, we can realize any symmetric group on a path with three colors. For example, the coloring \begin{center}
        \begin{tikzpicture}[scale = .75]
        \node (1) at (-9, 0) {};
        \node (2) at (-7, 0) {};
        \node (3) at (-5, 0) {};
        \node (4) at (-3, 0) {};
        \node (5) at (-1,0) {};
        \node (6) at (1, 0) {};
        \node (7) at (3, 0) {};
        \node (8) at (5, 0) {};
        \node (9) at (7, 0) {};
        \node (10) at (9,0) {};
        \draw[color = c1, thick] (1) -- (2);
        \draw[color = c2, thick] (3) -- (2);
        \draw[color = c1, thick] (3) -- (4);
        \draw[color = c2, thick] (5) -- (4);
        \draw[color = c1, thick] (5) -- (6);
        \draw[color = c2, thick] (6) -- (7);
        \draw[color = c1, thick] (7) -- (8);
        \draw[color = c2, thick] (8) -- (9);
        \draw[color = c3, thick] (9) -- (10);
        \draw[color = black, fill] (1) circle (2pt);
        \draw[color = black, fill] (2) circle (2pt);
        \draw[color = black, fill] (3) circle (2pt);
        \draw[color = black, fill] (4) circle (2pt);
        \draw[color = black, fill] (5) circle (2pt);
        \draw[color = black, fill] (6) circle (2pt);
        \draw[color = black, fill] (7) circle (2pt);
        \draw[color = black, fill] (8) circle (2pt);
        \draw[color = black, fill] (9) circle (2pt);
        \draw[color = black, fill] (10) circle (2pt);
    \end{tikzpicture}
    \end{center} has corresponding coloring group $S_{10}$.
\end{example}

In order to realize $A_n$ as a coloring group, all colors must appear on an even number of edges. In particular, this means that the order $n$ of our tree must be odd. It turns out that we can realize $A_n$ as a coloring group on a tree whenever $n>5$ and $n$ is odd. \begin{example} Color the path graph $P_{2k+1}$ by $1,2,\dots,k,1,2,\dots,k$. Then $(\tau_1\tau_k)^2$ is a three-cycle. Label the vertices so that $(\tau_1\tau_k)^2=(1,2,3)$. Then any product $\sigma$ of all the generators which ends in $\tau_k\tau_1\tau_2$ sends $1\to2\to3$. Label the remaining vertices so that $\sigma(i)=i+1$. The long cycle $\sigma$ and the three-cycle $(\tau_1\tau_k)^2$ generate all of $A_n$. For example, \begin{center}
        \begin{tikzpicture}[scale = .75]
        \node (1) at (-9, 0) {};
        \node (2) at (-7, 0) {};
        \node (3) at (-5, 0) {};
        \node (4) at (-3, 0) {};
        \node (5) at (-1,0) {};
        \node (6) at (1, 0) {};
        \node (7) at (3, 0) {};
        \node (8) at (5, 0) {};
        \node (9) at (7, 0) {};
        \draw[color = c1, thick] (1) -- (2);
        \draw[color = c2, thick] (3) -- (2);
        \draw[color = c3, thick] (3) -- (4);
        \draw[color = c4, thick] (5) -- (4);
        \draw[color = c1, thick] (5) -- (6);
        \draw[color = c2, thick] (6) -- (7);
        \draw[color = c3, thick] (7) -- (8);
        \draw[color = c4, thick] (8) -- (9);
        \draw[color = black, fill] (1) circle (2pt);
        \draw[color = black, fill] (2) circle (2pt);
        \draw[color = black, fill] (3) circle (2pt);
        \draw[color = black, fill] (4) circle (2pt);
        \draw[color = black, fill] (5) circle (2pt);
        \draw[color = black, fill] (6) circle (2pt);
        \draw[color = black, fill] (7) circle (2pt);
        \draw[color = black, fill] (8) circle (2pt);
        \draw[color = black, fill] (9) circle (2pt);
    \end{tikzpicture}
    \end{center} has corresponding coloring group $A_9$.

    Note that this construction fails for $n=5$ because, as we'll see in the following example, the coloring we obtain must be $D_5$.
\end{example} 

We can construct $D_n$ on the path graph with the minimal proper edge coloring.

\begin{example}
    The proper edge coloring of $P_n$ with two colors yields the dihedral group $D_n$. A nice way to see this is to consider the vertices of a regular $n$-gon and connect them into a path in a zig-zag fashion like so:

    
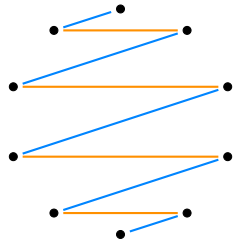
\begin{center}
        \begin{tikzpicture}[scale = .75]
        \node (1) at (0, 2) {};
        \node (2) at (-1.18, 1.62) {};
        \node (3) at (1.18, 1.62) {};
        \node (4) at (-1.9, .62) {};
        \node (5) at (1.9, .62) {};
        \node (6) at (-1.9, -.62) {};
        \node (7) at (1.9, -.62) {};
        \node (8) at (-1.18, -1.62) {};
        \node (9) at (1.18, -1.62) {};
        \node (10) at (0,-2) {};
        \draw[color = c1, thick] (1) -- (2);
        \draw[color = c2, thick] (3) -- (2);
        \draw[color = c1, thick] (3) -- (4);
        \draw[color = c2, thick] (5) -- (4);
        \draw[color = c1, thick] (5) -- (6);
        \draw[color = c2, thick] (6) -- (7);
        \draw[color = c1, thick] (7) -- (8);
        \draw[color = c2, thick] (8) -- (9);
        \draw[color = c1, thick] (9) -- (10);
        \draw[color = black, fill] (1) circle (2pt);
        \draw[color = black, fill] (2) circle (2pt);
        \draw[color = black, fill] (3) circle (2pt);
        \draw[color = black, fill] (4) circle (2pt);
        \draw[color = black, fill] (5) circle (2pt);
        \draw[color = black, fill] (6) circle (2pt);
        \draw[color = black, fill] (7) circle (2pt);
        \draw[color = black, fill] (8) circle (2pt);
        \draw[color = black, fill] (9) circle (2pt);
        \draw[color = black, fill] (10) circle (2pt);
    \end{tikzpicture}
    \end{center}

    From this illustration, we can clearly see that the two generators correspond to two reflections that generate $D_n$.
\end{example}

By the proof of \cref{thm:signed_permutations}, we can construct the group $B_m$ of signed permutations on a tree of order $n=2m$ by joining two identical trees whose corresponding coloring group is the symmetric group $S_m$ with an edge of a unique color.

\begin{example}
    Let $(T,\kappa)$ be the proper edge coloring of the star graph on $n$ vertices. Connect two copies of $T$ at their centers by a new edge colored $n+1$. The resulting color group is $B_n$. For example, the coloring group of 
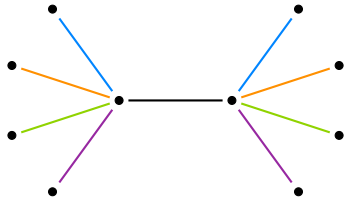
\begin{center}
         \begin{tikzpicture}[scale = .75]
        \node (1) at (-1, 0) {};
        \node (2) at (-2.18, 1.62) {};
        \node (3) at (2.18, 1.62) {};
        \node (4) at (-2.9, .62) {};
        \node (5) at (2.9, .62) {};
        \node (6) at (-2.9, -.62) {};
        \node (7) at (2.9, -.62) {};
        \node (8) at (-2.18, -1.62) {};
        \node (9) at (2.18, -1.62) {};
        \node (10) at (1,0) {};
        \draw[color = c1, thick] (1) -- (2);
        \draw[color = c2, thick] (1) -- (4);
        \draw[color = c3, thick] (1) -- (6);
        \draw[color = c4, thick] (1) -- (8);
        \draw[color = c1, thick] (10) -- (3);
        \draw[color = c2, thick] (10) -- (5);
        \draw[color = c3, thick] (10) -- (7);
        \draw[color = c4, thick] (10) -- (9);
        \draw[color = black, thick] (1) -- (10);
        \draw[color = black, fill] (1) circle (2pt);
        \draw[color = black, fill] (2) circle (2pt);
        \draw[color = black, fill] (3) circle (2pt);
        \draw[color = black, fill] (4) circle (2pt);
        \draw[color = black, fill] (5) circle (2pt);
        \draw[color = black, fill] (6) circle (2pt);
        \draw[color = black, fill] (7) circle (2pt);
        \draw[color = black, fill] (8) circle (2pt);
        \draw[color = black, fill] (9) circle (2pt);
        \draw[color = black, fill] (10) circle (2pt);
    \end{tikzpicture}
    \end{center} is $B_4$.
\end{example}

\section{Coloring Groups on Small Trees}\label{sec:data}

The following table lists the coloring groups presented by proper edge colorings of trees up to 12 vertices computed using Sage\footnote{Code available at \url{https://github.com/wilsoa/Coloring-Groups}} \cite{sagemath}. Groups isomorphic to $S_n, A_n, D_n$ or $B_m$ are omitted as their constructions are detailed above. The description in each row comes from GAP's \cite{GAP4} \texttt{structure\_description} method. A sample coloring realizing the group is given in the table either as a sequence of edge colors on a path or as a reference to another example if it is not possible to realize that group on a path.

\begin{center}
    \begin{tabular}{c|c|c|c|c}
        Degree & Description & Order & Primitive? & Coloring\\\hline
        7 & $GL(3,2)$ & 168 & Yes & \Cref{ex:gl32} \\ 
        9 & $(C_3 \times C_3 \times C_3) \rtimes S4$ & 648 & No & \Cref{ex:648}\\
        10& $(C_5 \times C_5) \rtimes D_4$ & 200 & No &1,2,1,3,1,2,1,3,1\\
         & $C_2 \times S_5$ & 240 & No & 1,2,1,3,1,3,1,2,1\\
        & $(A_5 \times A_5) \rtimes (C_2 \times C_2)$ & 14400 & No & 1,2,1,2,1,3,1,3,1\\
        12 & $(C_3 \times C_3) \rtimes ((C_2)^4\rtimes C_2)$ & 288 & No & 1,2,1,3,1,2,1,3,1,2,1\\
        & $(C_2 \times C_2) \rtimes ((C_2)^4 \rtimes D_6)$ & 768 & No & 1,2,3,2,1,2,1,2,3,2,1\\
        & $(C_3 \times C_3 \times C_3) \rtimes (C_2 \times S_4)$ & 1296 & No & 1,2,3,2,3,2,1,2,1,2,3\\
        & $(C_4 \times C_4 \times C_4) \rtimes S_4$ & 1536 & No & \Cref{ex:symmetric_tree} \\
        & $(C_2)^2 \rtimes ((C_2)^4 \rtimes ((S_3)^2 \rtimes C_2))$ & 4608 & No & 1,2,1,3,1,3,1,3,1,2,1\\
        & $(C_3)^4 \rtimes ((C_4 \times C_2) \rtimes D_4)$ & 5184 & No & \Cref{ex:5184} \\
        & $({C_2})^6 \rtimes (((C_3)^2 \rtimes C_3) \rtimes (C_2)^2)$ & 6912 & No & \Cref{ex:imprimitive_tree} \\
        & $(C_3)^4 \rtimes ((C_2)^3 \rtimes S_4)$ & 15552 & No & 1,2,3,4,3,2,1,4,1,2,3\\
        & $((C_2)^5 \rtimes A_6) \rtimes C_2$ & 23040 & No & \Cref{ex:23040}\\
        & $(C_3)^4 \rtimes ((C_2)^4 \rtimes S_4)$ & 31104 & No & \Cref{ex:31104} \\
        & $(C_2)^6 \rtimes ((C_3)^3 \rtimes (C_2 \times S_4))$ & 82944 & No & \Cref{ex:82944}\\
        & $(A_6 \times A_6) \rtimes D_4$ & 1036800 & No & 1,2,1,2,1,3,1,3,1,3,1
    \end{tabular}
\end{center}

Finally, we collect example colorings for groups in the above table which have not yet appeared.

\begin{example}\label{ex:648} Coloring realizing the group of degree 9 and order 648:
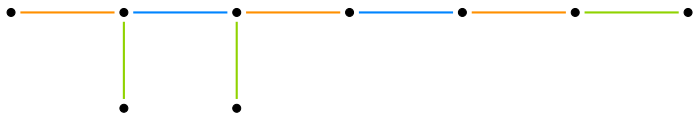
\begin{center}
    \begin{tikzpicture}[scale = .75]
        \node (1) at (-9, 0) {};
        \node (2) at (-7, 0) {};
        \node (3) at (-5, 0) {};
        \node (4) at (-3, 0) {};
        \node (5) at (-1,0) {};
        \node (6) at (1, 0) {};
        \node (7) at (-11, 0) {};
        \node (8) at (-9, -1.7) {};
        \node (9) at (-7, -1.7) {};
        \draw[color = c1, thick] (1) -- (2);
        \draw[color = c2, thick] (3) -- (2);
        \draw[color = c1, thick] (3) -- (4);
        \draw[color = c2, thick] (5) -- (4);
        \draw[color = c3, thick] (5) -- (6);
        \draw[color = c2, thick] (1) -- (7);
        \draw[color = c3, thick] (1) -- (8);
        \draw[color = c3, thick] (2) -- (9);
        \draw[color = black, fill] (1) circle (2pt);
        \draw[color = black, fill] (2) circle (2pt);
        \draw[color = black, fill] (3) circle (2pt);
        \draw[color = black, fill] (4) circle (2pt);
        \draw[color = black, fill] (5) circle (2pt);
        \draw[color = black, fill] (6) circle (2pt);
        \draw[color = black, fill] (7) circle (2pt);
        \draw[color = black, fill] (8) circle (2pt);
        \draw[color = black, fill] (9) circle (2pt);
    \end{tikzpicture}
\end{center}
\end{example}

\begin{example}\label{ex:5184} Coloring realizing the group of degree 12 and order 5184:
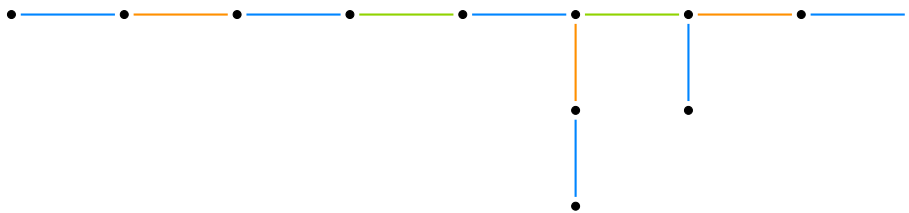
\begin{center}
        \begin{tikzpicture}[scale = .75]
        \node (1) at (-9, 0) {};
        \node (2) at (-7, 0) {};
        \node (3) at (-5, 0) {};
        \node (4) at (-3, 0) {};
        \node (5) at (-1,0) {};
        \node (6) at (1, 0) {};
        \node (7) at (3, 0) {};
        \node (8) at (5, 0) {};
        \node (9) at (7, 0) {};
        \node (10) at (1,-1.7) {};
        \node (11) at (1,-3.4) {};
        \node (12) at (3,-1.7) {};
        \draw[color = c1, thick] (1) -- (2);
        \draw[color = c2, thick] (3) -- (2);
        \draw[color = c1, thick] (3) -- (4);
        \draw[color = c3, thick] (5) -- (4);
        \draw[color = c1, thick] (5) -- (6);
        \draw[color = c3, thick] (6) -- (7);
        \draw[color = c2, thick] (7) -- (8);
        \draw[color = c1, thick] (8) -- (9);
        \draw[color = c2, thick] (6) -- (10);
        \draw[color = c1, thick] (10) -- (11);
        \draw[color = c1, thick] (7) -- (12);
        \draw[color = black, fill] (1) circle (2pt);
        \draw[color = black, fill] (2) circle (2pt);
        \draw[color = black, fill] (3) circle (2pt);
        \draw[color = black, fill] (4) circle (2pt);
        \draw[color = black, fill] (5) circle (2pt);
        \draw[color = black, fill] (6) circle (2pt);
        \draw[color = black, fill] (7) circle (2pt);
        \draw[color = black, fill] (8) circle (2pt);
        \draw[color = black, fill] (9) circle (2pt);
        \draw[color = black, fill] (10) circle (2pt);
        \draw[color = black, fill] (11) circle (2pt);
        \draw[color = black, fill] (12) circle (2pt);
    \end{tikzpicture}
\end{center}
\end{example}

\begin{example}\label{ex:23040} Coloring realizing the group of degree 12 and order 23040:
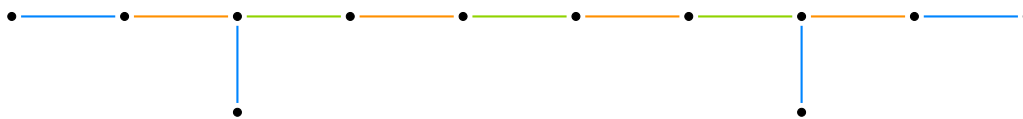
\begin{center}
        \begin{tikzpicture}[scale = .75]
        \node (1) at (-9, 0) {};
        \node (2) at (-7, 0) {};
        \node (3) at (-5, 0) {};
        \node (4) at (-3, 0) {};
        \node (5) at (-1,0) {};
        \node (6) at (1, 0) {};
        \node (7) at (3, 0) {};
        \node (8) at (5, 0) {};
        \node (9) at (7, 0) {};
        \node (10) at (9,0) {};
        \node (11) at (-5,-1.7) {};
        \node (12) at (5,-1.7) {};
        \draw[color = c1, thick] (1) -- (2);
        \draw[color = c2, thick] (3) -- (2);
        \draw[color = c3, thick] (3) -- (4);
        \draw[color = c2, thick] (5) -- (4);
        \draw[color = c3, thick] (5) -- (6);
        \draw[color = c2, thick] (6) -- (7);
        \draw[color = c3, thick] (7) -- (8);
        \draw[color = c2, thick] (8) -- (9);
        \draw[color = c1, thick] (9) -- (10);
        \draw[color = c1, thick] (3) -- (11);
        \draw[color = c1, thick] (8) -- (12);
        \draw[color = black, fill] (1) circle (2pt);
        \draw[color = black, fill] (2) circle (2pt);
        \draw[color = black, fill] (3) circle (2pt);
        \draw[color = black, fill] (4) circle (2pt);
        \draw[color = black, fill] (5) circle (2pt);
        \draw[color = black, fill] (6) circle (2pt);
        \draw[color = black, fill] (7) circle (2pt);
        \draw[color = black, fill] (8) circle (2pt);
        \draw[color = black, fill] (9) circle (2pt);
        \draw[color = black, fill] (10) circle (2pt);
        \draw[color = black, fill] (11) circle (2pt);
        \draw[color = black, fill] (12) circle (2pt);
    \end{tikzpicture}
\end{center}
\end{example}

\begin{example}\label{ex:31104} Coloring realizing the group of degree 12 and order 31104:
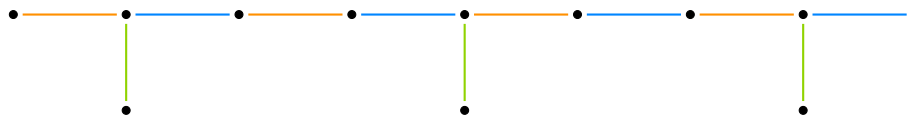
\begin{center}
        \begin{tikzpicture}[scale = .75]
        \node (1) at (-9, 0) {};
        \node (2) at (-7, 0) {};
        \node (3) at (-5, 0) {};
        \node (4) at (-3, 0) {};
        \node (5) at (-1,0) {};
        \node (6) at (1, 0) {};
        \node (7) at (3, 0) {};
        \node (8) at (-9, -1.7) {};
        \node (9) at (-11, 0) {};
        \node (10) at (5,0) {};
        \node (11) at (3,-1.7) {};
        \node (12) at (-3,-1.7) {};
        \draw[color = c1, thick] (1) -- (2);
        \draw[color = c2, thick] (3) -- (2);
        \draw[color = c1, thick] (3) -- (4);
        \draw[color = c2, thick] (5) -- (4);
        \draw[color = c1, thick] (5) -- (6);
        \draw[color = c2, thick] (6) -- (7);
        \draw[color = c3, thick] (1) -- (8);
        \draw[color = c2, thick] (1) -- (9);
        \draw[color = c1, thick] (7) -- (10);
        \draw[color = c3, thick] (7) -- (11);
        \draw[color = c3, thick] (4) -- (12);
        \draw[color = black, fill] (1) circle (2pt);
        \draw[color = black, fill] (2) circle (2pt);
        \draw[color = black, fill] (3) circle (2pt);
        \draw[color = black, fill] (4) circle (2pt);
        \draw[color = black, fill] (5) circle (2pt);
        \draw[color = black, fill] (6) circle (2pt);
        \draw[color = black, fill] (7) circle (2pt);
        \draw[color = black, fill] (8) circle (2pt);
        \draw[color = black, fill] (9) circle (2pt);
        \draw[color = black, fill] (10) circle (2pt);
        \draw[color = black, fill] (11) circle (2pt);
        \draw[color = black, fill] (12) circle (2pt);
    \end{tikzpicture}
\end{center}
\end{example}

\begin{example}\label{ex:82944} Coloring realizing the group of degree 12 and order 82944:
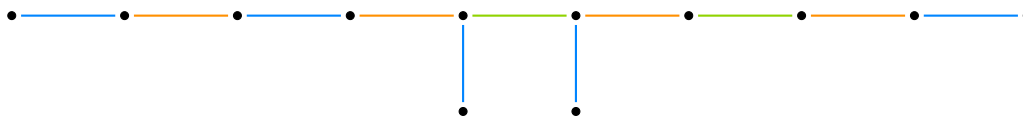
\begin{center}
        \begin{tikzpicture}[scale = .75]
        \node (1) at (-9, 0) {};
        \node (2) at (-7, 0) {};
        \node (3) at (-5, 0) {};
        \node (4) at (-3, 0) {};
        \node (5) at (-1,0) {};
        \node (6) at (1, 0) {};
        \node (7) at (3, 0) {};
        \node (8) at (5, 0) {};
        \node (9) at (7, 0) {};
        \node (10) at (9,0) {};
        \node (11) at (-1,-1.7) {};
        \node (12) at (1,-1.7) {};
        \draw[color = c1, thick] (1) -- (2);
        \draw[color = c2, thick] (3) -- (2);
        \draw[color = c1, thick] (3) -- (4);
        \draw[color = c2, thick] (5) -- (4);
        \draw[color = c3, thick] (5) -- (6);
        \draw[color = c2, thick] (6) -- (7);
        \draw[color = c3, thick] (7) -- (8);
        \draw[color = c2, thick] (8) -- (9);
        \draw[color = c1, thick] (9) -- (10);
        \draw[color = c1, thick] (5) -- (11);
        \draw[color = c1, thick] (6) -- (12);
        \draw[color = black, fill] (1) circle (2pt);
        \draw[color = black, fill] (2) circle (2pt);
        \draw[color = black, fill] (3) circle (2pt);
        \draw[color = black, fill] (4) circle (2pt);
        \draw[color = black, fill] (5) circle (2pt);
        \draw[color = black, fill] (6) circle (2pt);
        \draw[color = black, fill] (7) circle (2pt);
        \draw[color = black, fill] (8) circle (2pt);
        \draw[color = black, fill] (9) circle (2pt);
        \draw[color = black, fill] (10) circle (2pt);
        \draw[color = black, fill] (11) circle (2pt);
        \draw[color = black, fill] (12) circle (2pt);
    \end{tikzpicture}
\end{center}
    
\end{example}

\acknowledgements
The authors would like to thank Jonathan Bloom and Jessica Striker for their guidance and many, many helpful conversations. They would also like to thank Nate Lesnevich for his insights on an early version of this project.

\nocite{*}
\bibliographystyle{plain}
\bibliography{CopyEdited}
\label{sec:biblio}

\end{document}